\documentclass{amsart}

\usepackage{amsmath}
\usepackage{amssymb}
\usepackage{amsthm}
\usepackage{amscd}
\usepackage[all]{xy}

\newtheorem{thm}{Theorem}[section]
\newtheorem{prop}[thm]{Proposition}
\newtheorem{lem}[thm]{Lemma}
\newtheorem{cor}[thm]{Corollary}

\theoremstyle{definition}
\newtheorem{dfn}[thm]{Definition}

\theoremstyle{remark}
\newtheorem{rem}[thm]{Remark}
\newtheorem*{acknowledgement}{Acknowledgements}

\newcommand{\R}{\mathbb{R}}
\newcommand{\Z}{\mathbb{Z}}

\newcommand{\Fr}{\mathrm{Fr}}

\title{Smoothness filtration of the magnitude complex}

\author{Kiyonori Gomi}

\address{
Department of Mathematical Sciences,
Shinshu University, 
3--1--1 Asahi, Matsumoto, Nagano 390-8621,
Japan.}

\email{kgomi@math.shinshu-u.ac.jp}



\date{\today}

\begin{document}

\begin{abstract}
We introduce an intrinsic filtration to the magnitude chain complex of a metric space, and study basic properties of the associated spectral sequence of the magnitude homology. As an application, the third magnitude homology of the circle is computed. 
\end{abstract}

\maketitle

\tableofcontents


\section{Introduction}

The magnitude homology is a homology group of a chain complex constructed from any metric space. As a categorification of the magnitude of a metric space \cite{L}, this notion is first introduced to finite metric spaces associated to graphs with the shortest path length by Hepworth and Willerton \cite{H-W}, and then generalized to general metric spaces by Leinster and Shulman \cite{L-S}. Since the introduction, a number of features of the magnitude homology have been uncovered recently \cite{Hep,Jub,K-Y,Ott}. 

The purpose of this paper is to provide an intrinsic spectral sequence computing the magnitude homology. Spectral sequences are ubiquitous tools computing homology and cohomology of any kind, and are well developed in particular in algebraic topology (see \cite{B-T,K-T,Spa} for example). To provide our spectral sequence, we introduce a filtration to the chain complex giving the magnitude homology, based on an idea of ``smoothness'' of points in chains due to Kaneta and Yoshinaga \cite{K-Y}. Accordingly, we call it the smoothness filtration. Then we have the associated ``smoothness'' spectral sequence for the magnitude homology. Its properties to be studied in this paper are as follows:
\begin{itemize}
\item
We describe the $E^1$-term of the spectral sequence concretely. Then, using the notion of frames in \cite{K-Y}, we show that the $E^1$-term decomposes into the direct sum of subcomplexes.

\item
We show that the spectral sequence always degenerates at $E^4$.

\item
We provide criteria for $E^2$- and $E^3$-degeneracy. For instance, if there is no $4$-cut, then the spectral sequence degenerates at $E^2$.

\item
We show some vanishing results assuming the Menger convexity and an additional condition.

\end{itemize}

As an application of the spectral sequence, we compute the third magnitude homology of the circle of radius $r$. Its magnitude homology up to degree $2$ has been determined completely \cite{K-Y,L-S}, whereas the other cases seem to be not yet. In degree $3$, the vanishing of the magnitude homology is shown in some cases \cite{K-Y}.  We prove that this vanishing is always the case. (See Section \ref{sec:circle} for more detail.) This application proves that our spectral sequence is useful for the computation of the magnitude homology groups. In a subsequent paper \cite{G}, the spectral sequence is further applied to a complete description of the third magnitude homology of a so-called geodesic metric space.

\medskip

Provided the magnitude homology, there is also its dual, i.e.\ the magnitude cohomology. This is anticipated in \cite{H-W}, and is investigated in \cite{Hep} together with its non-commutative product. Dualizing the filtration of the magnitude chain complex, we get a filtration of the magnitude cochain complex. 

The associated spectral sequence for the magnitude cohomology has properties parallel to those of the spectral sequence for the magnitude homology, and further has a compatibility with the non-commutative product structure. However, these facts can be shown rather easily, so that their detail will not be given in this paper.

\bigskip

The outline of this paper is as follows. In Section \ref{sec:preliminary}, we recall some basic setup for the magnitude homology. In Section \ref{sec:filtration}, we introduce our filtration to the magnitude chain complex, and study the associated spectral sequence. At the end of this section, some sample computations of the spectral sequence are supplied. Then, in Section \ref{sec:circle}, we carry out the computation of the third magnitude homology of the circle applying our spectral sequence. 

\bigskip

\begin{acknowledgement}
I would like to thank Yasuhide Numata and Masahiko Yoshinaga for discussions. I would also like to thank Yuzhou Gu for pointing out a mistake in an earlier version of this paper. This work is supported by JSPS Grant-in-Aid for Scientific Research on Innovative Areas "Discrete Geometric Analysis for Materials Design": Grant
Number JP17H06462.
\end{acknowledgement}


\section{Preliminary}
\label{sec:preliminary}

As preliminary, we introduce some basic notions about magnitude complexes. For more detail, we refer to \cite{H-W,K-Y,L-S}, although we are using some different notations.

\medskip

Throughout this section, we choose and fix a metric space $(X, d)$.


\subsection{Chains of points}

For a non-negative integer $n$, an \textit{$n$-chain} $\langle x_0, \cdots, x_n \rangle$ of points on $X$ means a sequence $x_0, \cdots, x_n$ of $n+1$ points on $X$. We say that an $n$-chain $\langle x_0, \cdots, x_n \rangle$ is \textit{proper} if $x_0 \neq x_1 \neq x_2 \neq \cdots \neq x_{n-1} \neq x_n$. We let $\widehat{P}_n(X) = \widehat{P}_n$ be the set of $n$-chains, and $P_n(X) = P_n \subset \widehat{P}_n$ the subset of proper $n$-chains. Given an $n$-chain $\gamma = \langle x_0, \cdots, x_1 \rangle \in P_n$, we say that \textit{a point $x$ is in $\gamma$}, or \textit{$x$ belongs to $\gamma$}, when $x$ is one of the points $x_i$ in the sequence $x_0, \cdots, x_n$.

\medskip

We define the \textit{length} of an $n$-chain $\langle x_0, \cdots, x_n \rangle \in \widehat{P}_n$ to be
$$
\ell(\langle x_0, \cdots, x_n \rangle)
= d(x_0, x_1) + d(x_1, x_2) + \cdots + d(x_{n-1}, x_n).
$$
We write $\widehat{P}^\ell_n(X, d) = \widehat{P}^\ell_n(X) = \widehat{P}^\ell_n \subset \widehat{P}_n$ for the subset of $n$-chains of length $\ell$, and $P^\ell_n(X, d) = P^\ell_n(X) = P^\ell_n \subset P_n$ for that of proper $n$-chains of length $\ell$.

\medskip

Given three points $x, y, z \in X$, we say that \textit{$y$ is between $x$ and $z$}, and write $x < y < z$, when $x \neq y \neq z$ and $d(x, y) + d(y, z) = d(x, z)$ hold true. In the case that $y$ is not between $x$ and $z$, we write $x \not< y \not< z$. A point $x_i$ in a proper $n$-chain $\gamma = \langle x_0, \cdots, x_n \rangle$ is said to be \textit{smooth} if $i \neq 0, n$ and $x_{i-1} < x_i < x_{i+1}$. A point $x_i$ in $\gamma$ is \textit{singular} when $i = 0, 1$ or $x_{i-1} \not< x_i \not< x_{i+1}$. We follow \cite{K-Y} for the terms ``smooth'' and ``singular'', but the same notions are termed differently (``straight'' and ``crooked'') in \cite{Jub}.


\subsection{Magnitude complex}

For $\langle x_0, \cdots, x_n \rangle \in P_n(X)$ and $i = 0, \cdots, n$, we put
$$
\partial_i \langle x_0, \cdots, x_n \rangle
= \langle x_0, \cdots, x_{i-1}, x_i, \cdots, x_n \rangle.
$$
Notice that $\ell(\partial_i \gamma) \le \ell(\gamma)$ for any $\gamma \in P_n(X)$, due to the triangle inequality. More precisely, $\ell(\partial_i \gamma) < \ell(\gamma)$ if and only if the $i$th point $x_i$ in $\gamma = \langle x_0, \cdots, x_n \rangle$ is singular, and $\ell(\partial_i \gamma) = \ell(\gamma)$ if and only if $x_i$ in $\gamma$ is smooth. 

\medskip

We denote by $C^\ell_n(X, d) = C^\ell(X) = C^\ell_n$ the free abelian group generated on proper $n$-chains $\gamma \in P^\ell_n(X)$ of length $\ell$. We define a homomorphism $\partial : C^\ell_n \to C^\ell_{n-1}$ by the linear extension of
$$
\partial \langle x_0, \cdots, x_n \rangle
= \sum_{\stackrel{i = 1, \cdots, n-1}{x_{i-1} < x_i < x_{i+1}}}
(-1)^i \partial_i \langle x_0, \cdots, x_n \rangle.
$$
It can be shown that $\partial\partial = 0$, so that $(C^\ell_*, \partial)$ is a chain complex. This is called the \textit{magnitude complex of $(X, d)$}, and its homology group $H^\ell_*(X, d) = H^\ell_*(X) = H^\ell_*$ is the magnitude homology of $(X, d)$ (with the characteristic length $\ell$).

\medskip

In the case that $\ell < 0$, the magnitude homology is apparently trivial. In the case that $\ell = 0$, we immediately get \cite{L-S}
$$
H^0_n(X)
= C^0_n(X)
= 
\left\{
\begin{array}{ll}
\bigoplus_{x \in X}
\Z \langle x \rangle, & (n = 0) \\
0. & (n \neq 0)
\end{array}
\right.
$$
Therefore we are mainly interested in the case that $\ell > 0$.

\medskip

We may say that $x_0$ and $x_n$ are the \textit{end points} of $\gamma = \langle x_0, \cdots, x_n \rangle \in P^\ell_n$. If $x_j$ is a smooth point in $\gamma$, then $\partial_j\gamma \in P^\ell_{n-1}$ also has the end points $x_0$ and $x_n$. Thus, given two points $x, y \in X$, we have a subcomplex $C^\ell_*(X; x, y) = C^\ell_*(x, y) \subset C^\ell_*$ of proper $n$-chains with specified endpoints, and also the decompositions
\begin{align*}
C^\ell_*(X) &= \bigoplus_{x, y \in X} C^\ell_*(X; x, y), &
H^\ell_*(X) &= \bigoplus_{x, y \in X} H^\ell_*(X; x, y),
\end{align*}
which often reduce our consideration.

\medskip

\begin{rem}
Suppose that $(X, d)$ is a finite metric space. The metric space is said to be \textit{positive} \cite{L,Mar} if the matrix $Z = (e^{-d(x, y)})$ indexed by points $x, y \in X$ is positive definite. Then the \textit{magnitude} \cite{L,L-S} of the metric space is defined to be the sum $\mathrm{Mag}(X, d) = \sum_{x, y \in X} M(x, y)$ of the entries in the matrix $Z^{-1} = M = (M(x, y))$ inverse to $Z$. The crucial fact \cite{H-W,L-S} is the formula justifying that ``the magnitude homology is the categorification of the magnitude'':
$$
\mathrm{Mag}(X, d) = \sum_{\ell \ge 0} q^\ell
\sum_{n \ge 0} (-1)^n \mathrm{rank}H^\ell_n(X),
$$
where $q = 1/e$. One can further show that 
$$
M(x, y) = \sum_{\ell \ge 0} q^\ell
\sum_{n \ge 0} (-1)^n \mathrm{rank}H^\ell_n(x, y).
$$
Thus, if all the direct summands $H^\ell_n(x, y)$ of the magnitude homology $H^\ell_n(X, d)$ are known, then the original distance function $d(x, y)$ of the positive and finite metric space is recovered as $d(x, y) = -\log Z(x, y)$, where $Z(x, y)$ is the $(x, y)$-entry of the matrix $M^{-1} = Z = (Z(x, y))$ inverse to $M = (M(x, y))$ whose component $M(x, y)$ is defined by applying the formula above to $H^\ell_n(x, y)$. A more sophisticated recovery theorem by using the magnitude cohomology ring is known \cite{Hep}.
\end{rem}


\subsection{Persistence of singular point}

For the introduction of our filtration, we show a simple lemma.

\begin{lem} \label{lem:singular_pts_persist_in_boundary}
Let $\gamma = \langle x_0, \cdots, x_n \rangle \in P_n(X)$ be a proper $n$-chain $\gamma = \langle x_0, \cdots, x_n \rangle \in P_n(X)$ with $n \ge 2$, and $x_j$ a smooth point in $\gamma$. If $x_i$ is a singular point in $\gamma$, then $x_i$ is also a singular point in $\partial_j \gamma$.
\end{lem}

\begin{proof}
If $i = 0, \cdots, j - 2$ or $i = j + 2, \cdots, n$, then $x_i$ is clearly a singular point in $\partial_j \gamma$. Consider the case that $i = j - 1$. By assumption, we have 
\begin{align*}
d(x_{j-1}, x_j) + d(x_j, x_{j+1}) &= d(x_{j-1}, x_{j+1}), \\
d(x_{j-2}, x_{j-1}) + d(x_{j-1}, x_j) &> d(x_{j-2}, x_j).
\end{align*}
Using these assumptions and the triangle inequality, we get
\begin{align*}
d(x_{j-2}, x_{j-1}) + d(x_{j-1}, x_{j+1})
&=
d(x_{j-2}, x_{j-1}) + d(x_{j-1}, x_j) + d(x_j, x_{j+1}) \\
&>
d(x_{j-2}, x_j) + d(x_j, x_{j+1}) \\
&\ge
d(x_{j-2}, x_{j+1}).
\end{align*}
Hence $x_{j-2} \not< x_{j-1} \not< x_{j+1}$ and $x_{j-1}$ is singular in $\partial_j\gamma$. In the same vein, we can also prove that $x_{j-1} \not< x_{j+1} \not< x_{j+2}$ and $x_{j+1}$ is singular in $\partial_j\gamma$.
\end{proof}

Notice that, for a smooth point $x_j$ in $\gamma$ given, the other smooth point $x_i \neq x_j$ in $\gamma$ may turn into a singular point in $\partial_j\gamma$. For example, let $x_0 < x_1 < x_2 < x_3$ be a \textit{$4$-cut}, namely, a sequence of four points on $(X, d)$ such that we have $x_0 < x_1 < x_2$ and $x_1 < x_2 < x_3$, but 
\begin{align*}
d(x_0, x_3) &<
d(x_0, x_1) + d(x_1, x_2) + d(x_2, x_3) \\
&\quad 
= d(x_0, x_2) + d(x_2, x_3) = d(x_0, x_1) + d(x_1, x_3).
\end{align*}
In this case, both $x_1$ and $x_2$ are smooth points in $\langle x_0, x_1, x_2, x_3 \rangle$. However, $x_1$ becomes a singular point in $\partial_2\langle x_0, x_1, x_2, x_3 \rangle = \langle x_0, x_1, x_3 \rangle$.

\medskip

\begin{rem}
As an analogy of Lemma \ref{lem:singular_pts_persist_in_boundary}, we can show the following: Let $x_j$ be a smooth point in a proper $n$-chain $\langle x_0, \cdots, x_n \rangle$. Suppose that $y$ is a point on $X$ such that $x_{j-1} < y < x_j$. Then $x_j$ is still smooth in $\langle x_0, \cdots, x_{j-1}, y, x_j, \cdots x_n \rangle$.
\end{rem}


\section{Smoothness filtration}
\label{sec:filtration}

We let $(X, d)$ be a metric space, and $\ell > 0$ a positive real number.


\subsection{Smoothness filtration}

For a proper $n$-chain $\gamma \in P^\ell_n$, we write $\sigma(\gamma)$ for the number of smooth points in $\gamma$.

\begin{dfn}
Let $p \in \Z$ be an integer given.
\begin{itemize}
\item[(a)]
We define $F_pP^\ell_n \subset P^\ell_n$ to be the subset consisting of proper $n$-chains $\gamma$ of length $\ell$ whose numbers of smooth points are less than or equal to $p$, that is, $\sigma(\gamma) \le p$
$$
F_pP^\ell_n = \{ \gamma \in P^\ell_n |\ \sigma(\gamma) \le p \}.
$$

\item[(b)]
We define the subgroup $F_pC^\ell_n(X, d) = F_pC^\ell_n(X) = F_pC^\ell_n$ of $C^\ell_n$ to be the free abelian group generated on $F_pP^\ell_n$
$$
F_pC^\ell_n(X, d) = \bigoplus_{\gamma \in F_pP^\ell_n}\Z \gamma.
$$
\end{itemize}
\end{dfn}

By Lemma \ref{lem:singular_pts_persist_in_boundary}, each $F_pC^\ell_* \subset C^\ell_*$ is a subcomplex. Moreover, we have an increasing filtration of the magnitude chain complex
$$
\cdots \subset 
F_{p-1}C^\ell_* \subset 
F_pC^\ell_* \subset 
F_{p+1}C^\ell_* \subset 
\cdots,
$$
which we shall call the \textit{smoothness filtration}. By definition, we have $F_{p}C^\ell_* = 0$ if $p \le -1$. Since any proper $n$-chain with $n \ge 1$ has at least two singular points as its endpoints, we also have $C^\ell_n = F_{p}C^\ell_n$ for $p \ge n-1$. It is obvious that each subcomplex $C^\ell_*(x, y) \subset C^\ell_*(X)$ has its own filtration $F_pC^\ell_*(x, y)$.


\subsection{Spectral sequence}

Once a chain complex is endowed with an increasing filtration, a general argument yields the associated spectral sequence (see \cite{B-T,K-T,Spa} for instance). In our setup, we have the smoothness filtration $F_pC^\ell_*$ of the magnitude complex $C^\ell_*$. We thus have the associated \textit{smoothness spectral sequence} $E^r_{p, q}$. Namely, we have for each non-negative integer $r$ a chain complex $(E^r_{p, q}, d^r)$
$$
\cdots \overset{d^r}{\longleftarrow}
E^r_{p - r, q + r - 1} \overset{d^r}{\longleftarrow}
E^r_{p, q} \overset{d_r}{\longleftarrow}
E^r_{p + r, q - r + 1} \overset{d^r}{\longleftarrow}
\cdots,
$$
and its homology 
$$
E^{r+1}_{p, q}
= \mathrm{Ker}[d^r : E^r_{p, q} \to E^r_{p - r, q + r - 1}]/
\mathrm{Im}[d^r : E^r_{p + r, q - r + 1} \to E^r_{p, q}]
$$
produces the next chain complex $(E^{r+1}_{p, q}, d^{r+1})$. The filtration is clearly first quadrant, and is exhausting as observed. Hence the spectral sequence converges to the graded quotient of $H^\ell_*$ with respect to a filtration: 
\begin{itemize}
\item
For each $p, q$, we have a positive number $r_\infty$ such that $E^r_{p, q} = E^{r_\infty}_{p, q}$ for all $r \ge r_\infty$. We write $E^\infty_{p, q} = E^{r_\infty}_{p, q}$.

\item
Let $F_pH^\ell_n \subset H^\ell_n$ be the subgroup
$$
F_pH^\ell_n =
\mathrm{Im}[H_n(F_pC^\ell_*, \partial) \to H^\ell_n]
$$
given as the image of the homomorphism induced from the inclusion $F_pC^\ell_* \subset C^\ell_*$. These subgroups form a filtration
$$
0 = F_{-1}H^\ell_n \subset
\cdots \subset
F_{p-1}H^\ell_n \subset
F_pH^\ell_n \subset
F_{p+1}H^\ell_n \subset
\cdots \subset
F_{n-1}H^\ell_n = H^\ell_n,
$$
and there are short exact sequences
\begin{gather*}
0 \to F_{n-2}H^\ell_n \to H^\ell_n \to E^\infty_{n-1, 1} \to 0, \\
0 \to F_{n-3}H^\ell_n \to F_{n-2}H^\ell_n \to E^\infty_{n-2, 2} \to 0, \\
\vdots \\
0 \to F_1H^\ell_n \to F_2H^\ell_n \to E^\infty_{2, n-2} \to 0, \\
0 \to F_0H^\ell_n \to F_1H^\ell_n \to E^\infty_{1, n-1} \to 0,
\end{gather*}
where $F_0H^\ell_n = E^\infty_{0, n}$.
\end{itemize}

\bigskip

A property of the spectral sequence specific to the magnitude chain complex is as follows.

\begin{thm} \label{thm:general_triviality_of_differential}
The following holds true.
\begin{itemize}
\item[(a)]
The differential $d^0$ on $E^0_{p, q}$ is trivial, so that $E^1_{p, q} = E^0_{p, q}$.

\item[(b)]
For $r \ge 4$, the differential $d^r$ on $E^r_{p, q}$ is trivial, so that $E^4_{p, q} = E^\infty_{p, q}$.
\end{itemize}
\end{thm}

As a preliminary to the proof of this theorem, we show a simple lemma.

\begin{lem} \label{lem:estimate_of_smooth_points_in_boundary}
Let $\gamma \in P_n^\ell$ be a proper $n$-chain which contains $p$ smooth points: $\sigma(\gamma) = p$. Then its boundary $\partial \gamma \in C^\ell_{n-1}$ is a linear combination of proper $(n-1)$-chains whose numbers of smooth points are either $p-1$, $p-2$ or $p-3$.
\end{lem}

\begin{proof}
It is enough to verify the lemma for $n \ge 2$. Suppose that $x_j$ is a smooth point in $\gamma = \langle x_0, \cdots, x_n\rangle$. In the chain $\partial_j \gamma = \langle x_0, \cdots, x_{j-1}, x_{j+1}, \cdots, x_n \rangle$, the point $x_i$ with $i = 0, \cdots, j-2, j+2, \cdots, n$ is smooth (resp.\ singular) if $x_i$ is originally smooth (resp.\ singular) in $\gamma$. In other words, only $x_{j-1}$ and $x_{j+1}$ have the chance to become singular from smooth. Since the smooth point $x_j$ is already removed from $\gamma$, the number of smooth points in $\partial_j\gamma$ is $p-1$, $p-2$ or $p-3$, as stated.
\end{proof}

\begin{proof}
[The proof of Theorem \ref{thm:general_triviality_of_differential}]
(a)
By construction, $E^0_{p, q}$ is defined as
$$
E^0_{p, q} = F_pC^\ell_{p+q}/F_{p-1}C^\ell_{p+q},
$$
and the differential $d^0 : E^0_{p, q} \to E^0_{p, q-1}$ is the homomorphism
$$
F_pC^\ell_{p+q}/F_{p-1}C^\ell_{p+q} \to 
F_pC^\ell_{p+q-1}/F_{p-1}C^\ell_{p+q-1}
$$ 
induced from $\partial : F_pC^\ell_{p+q} \to F_pC^\ell_{p+q-1}$. As shown in  Lemma \ref{lem:estimate_of_smooth_points_in_boundary}, if $\gamma \in P^\ell_{p+q}$ is a $(p+q)$-chain with $p$ smooth points, then $\partial \gamma$ is a linear combination of $(p+q-1)$-chains with at most $p-1$ smooth points. This means $\partial \gamma \in F_{p-1}C^\ell_{p+q-1}$ and $d^0\gamma = 0$ in $E^0_{p-1, q}$. 

(b)
By the construction of the spectral sequence, we can represent an element in $E^4_{p, q}$ by $\gamma \in C^\ell_{p+q}$ such that $\partial \gamma$ is a linear combination of $(p+q-1)$-chains each of which contains at most $p-4$ smooth points. Then $d^4[\gamma] \in E^4_{p-4, q + 3}$ is represented by $\partial \gamma$. However, $\partial \gamma$ is a linear combination of $(p+q-1)$-chains which have at least $p-3$ smooth points, by Lemma \ref{lem:estimate_of_smooth_points_in_boundary}. Therefore $d^4[\gamma] = 0$ in $E^4_{p-4, q + 3}$. 
\end{proof}


\subsection{A description of $E^1$-term}

Since each proper chain has its definite number of smooth points, we can identify $E^1_{p, q} = F_pC^\ell_{p+q}/F_{p-1}C^\ell_{p+q}$ with the free abelian group generated on proper $(p+q)$-chains of length $\ell$ which contain exactly $p$ smooth points. Under this identification, we get a decomposition as an abelian group
$$
C^\ell_n = \bigoplus_{p + q = n} E^1_{p, q}.
$$
This decomposition is generally not compatible with $\partial$. Its compression induces the boundary map $d^1 : E^1_{p, q} \to E^1_{p-1, q}$. Concretely, if $\gamma \in P^\ell_{p+q}$ is a proper chain with $\sigma(\gamma) = p$, then
$$
d^1\gamma
= \sum_i (-1)^i \partial_i \gamma,
$$
where $i$ runs over $1, \cdots, n -1$ such that $\ell(\partial_i\gamma) = \ell(\gamma) = \ell$ and $\sigma(\partial_i\gamma) = \sigma(\gamma) - 1 = p - 1$. 

\medskip

In the description above, it is clear that $d^1$ preserves the singular points in proper chains. This fact suggests that a focus on singular points is a key to understand the chain complex $(E^1, d^1)$. For this aim, recall from \cite{K-Y} that the \textit{frame} $\Fr(\gamma)$ of a proper $n$-chain $\gamma = \langle x_0, \cdots, x_n \rangle \in P_n$ is the $(n - \sigma(\gamma))$-chain of all the singular points in $\gamma$, given by removing all the smooth points in $\gamma$. Notice that $\Fr(\gamma)$ may be no longer a proper chain. Recall also that the set of all (possibly improper) $q$-chains is denoted by $\widehat{P}_q$. 

\begin{dfn}
Given $\varphi \in \widehat{P}_q$ as well as $n$ and $\ell$, we define
\begin{align*}
P^{\ell}_n(\varphi) &= \{ \gamma \in P^\ell_n |\ \Fr(\gamma) = \varphi \}, &
C^{\ell}_n(\varphi) &= \bigoplus_{\gamma \in P^\ell_n(\varphi)} \Z \gamma.
\end{align*}
\end{dfn}

When $P^\ell_n(\varphi) = \emptyset$, we understand $C^\ell_n(\varphi) = 0$. Since a chain which can be a frame of another (proper) chain has at least two points, we have $P^\ell_n(\varphi) = \emptyset$ and $C^\ell_n(\varphi) = 0$ for $\varphi = \langle \varphi_0 \rangle \in \widehat{P}_0$. In general, the length of a proper chain $\gamma \in P^\ell_n(\varphi)$ has the bound $\ell(\gamma) \ge \ell(\varphi)$. Hence $C^\ell_n(\varphi) = 0$ for $\ell < \ell(\varphi)$. We have $C^\ell_n(\varphi) = 0$ as well for any $\varphi \in \widehat{P}_q$ and $n < q$. Actually, if there exists $\gamma \in P^\ell_n(\varphi)$ for $\varphi \in \widehat{P}_q$ given, then $n - \sigma(\gamma) = q$. This fact also shows $C^\ell_{p+q}(\varphi) \subset E^1_{p, q}$ for $\varphi \in \widehat{P}_q$.

\begin{thm} \label{thm:direct_sum_decomposition_of_E1}
Let $\ell > 0$ be a real number, and $q > 0$ an integer.
\begin{itemize}
\item[(a)]
For each $\varphi \in \widehat{P}_q$, the groups $C^\ell_{p+q}(\varphi)$ constitute a subcomplex $C^\ell_{* + q}(\varphi) = \bigoplus_{p} C^\ell_{p+q}(\varphi)$ of the chain complex $(E^1_{*, q}, d^1)$.

\item[(b)]
Moreover, we have a direct sum decomposition of the chain complex
$$
E^1_{p, q} = \bigoplus_{\varphi \in \widehat{P}_q}
C^\ell_{p+q}(\varphi).
$$
In particular, $E^1_{p, 0} = 0$ for all $p$.
\end{itemize}
\end{thm}

\begin{proof}
For (a), recall that $d^1$ preserves the singular points in proper chains. Thus, if $\gamma \in P^\ell_{p+q}(\varphi)$, then $d^1 \gamma$ is a linear combination of chains in $P^\ell_{p+q-1}(\varphi)$, namely $d^1\gamma \in C^\ell_{p+q-1}(\varphi)$. For (b), let $\gamma \in E^1_{p, q}$ be a proper $(p+q)$-chain such that $\ell(\gamma) =\ell$ and $\sigma(\gamma) = p$. Then we have $\Fr(\gamma) \in \widehat{P}_q$ and $\gamma \in P^\ell_{p+q}(\Fr(\gamma))$.
\end{proof}

\begin{thm} \label{thm:tensor_product_decomposition}
Let $\varphi = \langle \varphi_0, \cdots, \varphi_q \rangle \in \widehat{P}_q$ be given. Then there is an injective chain map from $(C^\ell_*(\varphi), d^1)$ to the direct sum of tensor products of the chain complexes $(C^{\ell_j}_*(\langle \varphi_j, \varphi_{j+1} \rangle), d^1)$ 
$$
\iota :\ 
C^\ell_n(\varphi)
\longrightarrow 
\bigoplus_{\ell = \sum_j \ell_j} \bigoplus_{n = \sum_j n_j}
C^{\ell_1}_{n_1}(\langle \varphi_0, \varphi_1 \rangle) \otimes \cdots \otimes
C^{\ell_{q}}_{n_{q}}(\langle \varphi_{q-1}, \varphi_q \rangle).
$$
\end{thm}

\begin{proof}
The injective homomorphism $\iota$ of abelian groups is given by
\begin{gather*}
\langle \varphi_0, x^1_1, \cdot\cdot, x^1_{n_1 - 1}, \varphi_1, 
x^2_1, \cdot\cdot, x^2_{n_2 - 1}, \varphi_2, \cdots,
\varphi_{q-1}, x^{q}_1, \cdot\cdot, x^{q}_{n_{q} - 1}, \varphi_q \rangle \\
\downarrow \\
\langle \varphi_0, x^1_1, \cdot\cdot, x^1_{n_1-1}, \varphi_1 \rangle \otimes
\langle \varphi_1, x^2_1, \cdot\cdot, x^2_{n_2-1}, \varphi_2 \rangle \otimes 
\cdots \otimes
\langle \varphi_{q-1}, x^q_1, \cdot\cdot, x^q_{n_q - 1}, \varphi_q \rangle.
\end{gather*}
This homomorphism can be seen to be a chain map with respect to $d^1$ on $C^\ell_*(\varphi)$ and a tensor product of $d^1$ on $C^{\ell_j}_*(\langle y_j, y_{j+1} \rangle)$ with appropriate signs.
\end{proof}

The injective homomorphism $\iota$ in Theorem \ref{thm:tensor_product_decomposition} fails to be surjective in the following case for example: Suppose that 
\begin{align*}
\gamma &= 
\langle \varphi_0, x_1, \cdots, x_{n-1}, \varphi_1 \rangle 
\in C^\ell_n(\langle \varphi_0, \varphi_1 \rangle), \\
\gamma' &= 
\langle \varphi_1, z_1, \cdots, z_{n'-1}, \varphi_2 \rangle 
\in C^{\ell'}_{n'}(\langle \varphi_1, \varphi_2 \rangle)
\end{align*}
are such that $x_{n-1} < \varphi_1 < z_1$. Then $\gamma \otimes \gamma'$ is not in the image of $\iota$, since 
$$
\langle \varphi_0, x_1, \cdots, x_{n-1}, 
\varphi_1, z_1, \cdots, z_{n'-1}, \varphi_2 \rangle
$$ 
does not belong to $C^{\ell + \ell'}_{n+n'}(\langle \varphi_0, \varphi_1, \varphi_2 \rangle)$. The injective chain map $\iota$ also fails to be a quasi-isomorphism generally, as will be seen in an example in \S\S\ref{subsec:example}.

\medskip

Notice that the range of $\iota$ has the direct sum decomposition with respect to the decomposition of the length $\ell$. This induces a further decomposition of $C^\ell_*(\varphi)$ through the injective map $\iota$.

\begin{dfn} \label{dfn:refined_direct_sum_decomposition}
Let $q \ge 1$ be an integer, and $\varphi = \langle \varphi_0, \cdots, \varphi_q \rangle \in \widehat{P}_q$ a $q$-chain. For positive real numbers $\ell_1, \cdots, \ell_{q}$ and non-negative integers $n_1, \cdots, n_{q}$, we define 
$$
C^{\ell_1, \cdots, \ell_{q}}_{n_1, \cdots, n_q}(\varphi)
\subset C^{\ell_1 + \cdots + \ell_{q}}_{n_1 + \cdots + n_{q}}(\varphi)
$$
to be the subgroup generated by the $(n_1 + \cdots + n_{q})$-chains of length $\ell_1 + \cdots + \ell_{q}$ 
$$
\langle \varphi_0, x^1_1, \cdots, x^1_{n_1 - 1}, \varphi_1, 
x^2_1, \cdots, x^2_{n_2 - 1}, \varphi_2, \cdots,
\varphi_{q-1}, x^{q}_1, \cdots, x^{q}_{n_{q} - 1}, \varphi_q \rangle
$$
such that $\ell(\langle \varphi_{j-1}, x^j_1, \cdots, x^j_{n_j - 1}, \varphi_j \rangle) = \ell_j$ for $j = 1, \cdots, q$. We also define
$$
C^{\ell_1, \cdots, \ell_{q}}_n(\varphi)
= \bigoplus_{n = n_1 + \cdots + n_{q}}
C^{\ell_1, \cdots, \ell_{q}}_{n_1, \cdots, n_q}(\varphi).
$$
\end{dfn}

As a result of Theorem \ref{thm:tensor_product_decomposition}, we get

\begin{cor} \label{cor:tensor_product_decomposition}
Let $q \ge 1$ be an integer, and $\varphi = \langle \varphi_0, \cdots, \varphi_q \rangle \in \widehat{P}_q$ a $q$-chain. For positive real numbers $\ell_1, \cdots, \ell_{q}$, we have a direct sum decomposition
$$
C^\ell_n(\varphi)
=
\bigoplus_{\ell = \sum_j \ell_j}
C^{\ell_1, \cdots, \ell_{q}}_{n}(\varphi)
= 
\bigoplus_{\ell = \sum_j \ell_j} \bigoplus_{n = \sum_j n_j}
C^{\ell_1, \cdots, \ell_{q}}_{n_1, \cdots, n_{q}}(\varphi),
$$
and $\iota$ in Theorem \ref{thm:tensor_product_decomposition} restricts to an injective chain map
$$
\iota : \ 
C^{\ell_1, \cdots, \ell_{q}}_n(\varphi)
\longrightarrow
\bigoplus_{n = \sum_j n_j}
C^{\ell_1}_{n_1}(\langle \varphi_0, \varphi_1 \rangle) \otimes \cdots \otimes
C^{\ell_{q}}_{n_{q}}(\langle \varphi_{q-1}, \varphi_q \rangle).
$$
\end{cor}

\medskip

Theorem \ref{thm:direct_sum_decomposition_of_E1} and Theorem \ref{thm:tensor_product_decomposition} are parallel to some results in \cite{K-Y}. To explain this, let $\varphi \in \widehat{P}_q$ be a $q$-chain, and focus on the following subset
$$
P^{\ell(\varphi)}_n(\varphi)
= \{ \gamma \in P^{\ell(\varphi)}_n |\ \Fr(\gamma) = \varphi \}
\subset P^{\ell(\varphi)}_n
$$
It can be seen that $P^{\ell(\varphi)}_n(\varphi) = \emptyset$ if $\varphi \not\in P_q \subset \widehat{P}_q$. By definition, an $n$-chain $\gamma \in P^{\ell(\varphi)}_n(\varphi)$ satisfies $\Fr(\gamma) = \varphi$ and $\ell(\gamma) = \ell(\Fr(\gamma))$. In \cite{K-Y}, a proper chain $\gamma \in P_n$ is called \textit{geodesically simple} if $\ell(\gamma) = \ell(\Fr(\gamma))$. Therefore $P^{\ell(\varphi)}_n(\varphi)$ consists of geodesically simple chains, and a geodesically simple $n$-chain $\gamma$ belongs to $P^\ell_n(\Fr(\gamma))$ with $\ell = \ell(\gamma) = \ell(\Fr(\gamma))$. A fact shown in \cite{K-Y} is that the group
$$
C^{\ell(\varphi)}_n(\varphi)
= \bigoplus_{\gamma \in P^{\ell(\varphi)}_n(\varphi)} \Z\gamma
$$
forms a complex $C^{\ell(\varphi)}_*(\varphi)$ with respect to $\partial$. Thus, the geodesically simple chains form a subcomplex of the magnitude chain complex, and this subcomplex admits the direct sum decomposition
$$
\bigoplus_{\stackrel{\gamma \in P^\ell}
{\mbox{\small{geodesically simple}}}}
\Z \gamma
= \bigoplus_{\varphi \in P} C^{\ell(\varphi)}_*(\varphi),
$$
where $P^\ell = \bigcup_n P^\ell_n$ and $P = \bigcup_q P_q$.

It holds that $\partial = d^1$ on $C^{\ell(\varphi)}_*(\varphi) \subset C^{\ell(\varphi)}_*$. Therefore one can think of the complex of geodesically simple chains as a subcomplex of $(E^1, d^1)$. Then the above decomposition of the complex of geodesically simple chains is compatible with Theorem \ref{thm:direct_sum_decomposition_of_E1}. It is known in \cite{K-Y} that each direct summand $C^{\ell(\varphi)}_*(\varphi)$ admits a tensor product decomposition as a chain complex with respect to $\partial = d^1$, whereas Theorem \ref{thm:tensor_product_decomposition} only gives a injective chain map into a tensor product. 

Thus, there are generally some differences between our results and those in \cite{K-Y}. However, in a setup which forces all the proper chains to be geodesically simple (to be considered shortly), Theorem \ref{thm:direct_sum_decomposition_of_E1} and Theorem \ref{thm:tensor_product_decomposition} essentially agree with decompositions given in \cite{K-Y}. It should be noticed here that there is a crucial result in \cite{K-Y} which is not generalized in this paper: This is a description of the chain complex $C^{d(x, y)}_*(\langle x, y \rangle)$ in terms of the order complex of a poset associated to the interval $\langle x, y \rangle$.


\subsection{$E^2$-degeneracy and absence of $4$-cuts}

Following \cite{K-Y}, we introduce a positive number $m_X$ as follows
$$
m_X = \mathrm{inf}
\{ \ell(\langle x_0, x_1, x_2, x_3 \rangle) |\
\mbox{$\langle x_0, x_1, x_2, x_3 \rangle$ is a $4$-cut of $X$} \}.
$$
In the case where $X$ has no $4$-cut, then we set $m_X = + \infty$.

\begin{thm} \label{thm:E2_degeneracy}
Let $(X, d)$ be a metric space, and $\ell$ a real number such that $0 < \ell < m_X$. Then the smoothness spectral sequence of $H^\ell_*(X)$ degenerates at the $E^2$-term: $E^\infty_{p, q} = E^2_{p, q}$. Further, we have the direct sum decomposition
$$
H^\ell_n(X) = E_{0, n}^\infty \oplus \cdots \oplus E_{n-1, 1}^\infty.
$$
\end{thm}

\begin{proof}
The assumption $\ell < m_X$ implies that there is no $4$-cut in the proper chains of length $\ell$, and all the proper chains are geodesically simple. Accordingly, if a proper $n$-chain $\gamma \in P^\ell_n$ contains $p$ smooth points, then its boundary $\partial \gamma$ is a linear combination of proper $(n-1)$-chains which contain exactly $p-1$ smooth points. (No smooth point turns into singular.) 

Now, in view of the construction of the spectral sequence, the condition for $\gamma \in E^1_{p, q} = F_pC^\ell_{p+q}/F_{p-1}C^\ell_{p+q}$ to satisfy $d^1\gamma = 0$ is that $\partial \gamma \in F_{p-2}C^\ell_{p+q-1}$. Then $\gamma$ represents an element of $E^2_{p, q}$, and its image under $d^2$, which is in $E^2_{p-2, q+1}$, is represented by $\partial \gamma \in F_{p-2}C_{p+q-1}$. But, $\gamma$ is a linear combination of proper chains containing $p$ smooth points, and $\partial \gamma$ is a linear combination of proper chains containing $p-1$ smooth points by the present assumption. Hence $\partial \gamma = 0$ in $F_{p-2}C_{p+q-1}$, and $d_2$ is trivial. In the same way, $\partial \gamma = 0$ in $F_{p - r}C_{p+q-1}$ for $r \ge 2$, and $d_r$ is trivial. This establishes $E^2_{p, q} = E^r_{p, q}$ for $r \ge 2$.

For the decomposition of $H^\ell_n$, we just observe that the decomposition 
$$
C^\ell_n = \bigoplus_p E^1_{p, n - p}
$$
is compatible with $\partial$ under the present assumption.
\end{proof}

As is mentioned already, if $\ell < m_X$, then all the computations about $(E^1, d^1)$ can be done in a way developed in \cite{K-Y}.


\subsection{$E^3$-degeneracy and absence of overlaps of $4$-cuts}

In the proof of Theorem \ref{thm:E2_degeneracy}, it is observed that a $4$-cut is related to the differential $d^2 : E^2_{p, q} \to E^2_{p-2, q + 1}$ of the $E^2$-terms. A similar consideration is possible for $d^3 : E^3_{p, q} \to E^3_{p-3, q + 1}$, the remaining differential which is possibly non-trivial.

\begin{lem} \label{lem:overlap_of_4cuts}
For a proper $4$-chain $\langle x_0, x_1, x_2, x_3, x_4 \rangle$, the following are equivalent:
\begin{itemize}
\item[(a)]
$x_1, x_2, x_3$ are smooth points in $\langle x_0, x_1, x_2, x_3, x_4 \rangle$, but $x_1, x_3$ are singular points in $\langle x_0, x_1, x_3, x_4 \rangle$.

\item[(b)]
$\langle x_0, x_1, x_2, x_3 \rangle$ and $\langle x_1, x_2, x_3, x_4 \rangle$ are $4$-cuts.

\item[(c)]
It holds that
\begin{align*}
d(x_0, x_1) + d(x_1, x_2) &= d(x_0, x_2), \\
d(x_1, x_2) + d(x_2, x_3) &= d(x_1, x_3), \\
d(x_2, x_3) + d(x_3, x_4) &= d(x_2, x_4), \\
d(x_0, x_1) + d(x_1, x_3) &> d(x_0, x_3), \\
d(x_1, x_3) + d(x_3, x_4) &> d(x_1, x_4).
\end{align*}
\end{itemize}
\end{lem}

\begin{proof}
The verification is straightforward.
\end{proof}

Let us define a positive number $n_X$ by
$$
n_X = \mathrm{inf}\{
\ell(\langle x_0, x_1, x_2, x_3, x_4 \rangle) |\
\mbox{$\langle x_0, x_1, x_2, x_3 \rangle$ and $\langle x_1, x_2, x_3, x_4 \rangle$ are $4$-cuts}
\}.
$$
We set $n_X = + \infty$, when there is no proper $4$-chains $\langle x_0, x_1, x_2, x_3, x_4 \rangle$ such that $\langle x_0, x_1, x_2, x_3 \rangle$ and $\langle x_1, x_2, x_3, x_4 \rangle$ are $4$-cuts. Since the definition of $n_X$ involves $4$-cuts, we get the inequality $m_X \le n_X$.

\begin{thm} \label{thm:E3_degeneracy}
Let $(X, d)$ be a metric space, and $\ell$ a real number such that $0 < \ell < n_X$. Then the smoothness spectral sequence of $H^\ell_*(X)$ degenerates at the $E^3$-term: $E^\infty_{p, q} = E^3_{p, q}$.
\end{thm}

\begin{proof}
Let $\gamma \in P^\ell_n$ be a proper $n$-chain which contains $p$ smooth points. Then, by the assumption $\ell < n_X$ and Lemma \ref{lem:overlap_of_4cuts}, the boundary $\partial \gamma$ is a linear combination of proper $(n-1)$-chains which contain at least $p-2$ smooth points. In other words, proper $(n-1)$-chains with $p-3$ smooth points or less never appear in the linear combination $\partial \gamma$. This means that $d^3 : E^3_{p, q} \to E^3_{p-3, q+2}$ is trivial, in view of the machinery of the spectral sequence.
\end{proof}


\subsection{Menger convexity and vanishing of $E^2_{0, 1}$}

A metric space $(X, d)$ is said to be \textit{Menger convex} if any distinct points $x, z \in X$ admit a point $y \in X$ such that $x < y < z$. 

\begin{prop}
Let $(X, d)$ be a Menger convex metric space, and $\ell > 0$ a real number.
\begin{itemize}
\item[(a)]
For any $1$-chain $\langle \varphi_0, \varphi_1 \rangle \in \widehat{P}_1$, we have $H^\ell_1(\langle \varphi_0, \varphi_1 \rangle) = 0$.

\item[(b)]
$E^2_{0, 1} = 0$.

\end{itemize}
\end{prop}

\begin{proof}
For (a), if $\varphi_0 = \varphi_1$, then $C^\ell_1(\langle \varphi_0, \varphi_1 \rangle) = 0$. Therefore we assume $\varphi_0 \neq \varphi_1$ and $\langle \varphi_0, \varphi_1 \rangle \in P_1 \subset \widehat{P}_1$. Then we have 
\begin{align*}
C^\ell_0(\langle \varphi_0, \varphi_1 \rangle) &= 0, &
C^\ell_1(\langle \varphi_0, \varphi_1 \rangle) &= 
\left\{
\begin{array}{ll}
\Z \langle \varphi_0, \varphi_1 \rangle, & (\ell = d(\varphi_0, \varphi_1)) \\
0, & (\ell \neq d(\varphi_0, \varphi_1))
\end{array}
\right.
\end{align*}
In the case that $\ell \neq d(\varphi_0, \varphi_1)$, there is nothing to prove. Therefore we assume $\ell = d(\varphi_0, \varphi_1)$. Then the basis $\langle \varphi_0, \varphi_1 \rangle$ of $C^\ell_1(\langle \varphi_0, \varphi_1 \rangle)$ is a cycle with respect to $d^1$. By the Menger convexity, there exists $c \in X$ such that $\varphi_0 < c < \varphi_1$. Thus $\langle \varphi_0, c, \varphi_1 \rangle \in C^\ell_2(\langle \varphi_0, \varphi_1 \rangle)$. We have
$$
-d^1 \langle \varphi_0, c, \varphi_1 \rangle 
= \langle \varphi_0, \varphi_1 \rangle,
$$
so that $H^\ell_1(\langle \varphi_0, \varphi_1 \rangle) = 0$. For (b), we note that Theorem \ref{thm:direct_sum_decomposition_of_E1} leads to the direct sum decomposition 
$$
E^2_{0, 1} = \bigoplus_{\varphi \in \widehat{P}_1}
H^\ell_1(\varphi),
$$
where $H^\ell_1(\varphi)$ is the homology of $(C^\ell_*(\varphi), d^1)$. Hence $E^2_{0, 1} = 0$ by (a).
\end{proof}

As a consequence of the proposition, we can reprove the following fact \cite{L-S}:

\begin{cor}
For a Menger convex space $(X, d)$ and a real number $\ell > 0$, we have $H^\ell_1(X) = 0$.
\end{cor}

\begin{proof}
In general, $E^r_{p, 0} = 0$ for all $p, r$, provided that $\ell > 0$. Under the assumption that $(X, d)$ is Menger convex, the $E^2$-term reads:
$$
\begin{array}{c|c|c|c|}
\hline
q = 1 & 0 & & \\
\hline
q = 0 & 0 & 0 & 0 \\
\hline
E^2_{p, q} & p = 0 & p = 1 & p = 2 
\end{array}
$$
Therefore $E^\infty_{p, q} = 0$ for any $p, q$ such that $p + q = 1$, and also $H^\ell_1(X) = 0$.
\end{proof}

With an assumption additional to the Menger convexity, we have another vanishing of $E^2$-terms.

\begin{prop} \label{prop:Menger_convexity_plus_alpha}
Let $(X, d)$ be a Menger convex metric space, and $\ell > 0$ a real number. Suppose that for any distinct points $x, z \in X$ and any real number $\epsilon > 0$, there exists a point $y \in X$ such that $x < y < z$ and $d(x, y) < \epsilon$. Then the following holds true for any integer $q \ge 2$.
\begin{itemize}
\item[(a)]
For any $\varphi = \langle \varphi_0, \cdots, \varphi_q \rangle \in \widehat{P}_q$, we have $H^\ell_q(\varphi) = 0$.

\item[(b)]
$E^2_{0, q} = 0$.
\end{itemize} 
\end{prop}

\begin{proof}
(a) 
If $\varphi \not\in P_q$, then $C^\ell_q(\varphi) = 0$. Hence we assume that $\varphi \in P_q \subset \widehat{P}_q$. In this case, we have
\begin{align*}
C^\ell_{q-1}(\varphi) &= 0, &
C^\ell_q(\varphi) 
&= 
\left\{
\begin{array}{ll}
\Z \varphi, & (\ell = \ell(\varphi)) \\
0. & (\ell \neq \ell(\varphi))
\end{array}
\right.
\end{align*}
As a result, it suffices to show that the cycle $\varphi \in C^\ell_q(\varphi)$ is a boundary of a $(q + 1)$-chain in $C^\ell_{q+1}(\varphi)$ with respect to $d^1$, provided that $\ell = \ell(\varphi)$. Under the assumption in the proposition, we claim that there exists $c \in X$ such that $\varphi_0 < c < \varphi_1$ and $c \not< \varphi_1 \not< \varphi_2$. To prove this claim, let us define a function $F : X \to \R$ by
$$
F(x) = d(x, \varphi_1) + d(\varphi_1, \varphi_2) - d(x, \varphi_2).
$$
We have $F(\varphi_0) > 0$, since $\varphi_0 \not< \varphi_1 \not< \varphi_2$. It is easy to get the inequality
$$
\lvert F(x) - F(x') \rvert
\le \lvert d(x, \varphi_1) - d(x', \varphi_1) \rvert
+ \lvert d(x, \varphi_2) - d(x', \varphi_2) \rvert
\le 2d(x, x')
$$
for any $x, x' \in X$. By the assumption in the proposition, there exists $c \in X$ such that $\varphi_0 < c < \varphi_1$ and $2 d(\varphi_0, c) < F(\varphi_0)$. Then we have $\lvert F(\varphi_0) - F(c) \rvert < F(\varphi_0)$. If $F(\varphi_0) < F(c)$, then we apparently have $0 < F(\varphi_0) < F(c)$. If $F(\varphi_0) > F(c)$, then $F(\varphi_0) > \lvert F(\varphi_0) - F(c) \vert = F(\varphi_0) - F(c)$ and hence $F(c) > 0$. This means the existence of $c \in X$ such that $\varphi_0 < c < \varphi_1$ and $c \not< \varphi_1 \not< \varphi_2$. Now, we have $\langle \varphi_0, c, \varphi_1, \cdots, \varphi_q \rangle \in C^\ell_{q+1}(\langle \varphi_0, \cdots, \varphi_q \rangle)$, whose boundary is
$$
d\langle \varphi_0, c, \varphi_1, \cdots, \varphi_q \rangle 
= - \langle \varphi_0, \varphi_1, \cdots, \varphi_q \rangle. 
$$
Therefore (a) is established.

(b) Theorem \ref{thm:direct_sum_decomposition_of_E1} gives us the direct sum decomposition
$$
E^2_{0, q} = \bigoplus_{\varphi \in \widehat{P}_q}
H^\ell_q(\varphi).
$$
Hence (b) follows from (a).
\end{proof}

\begin{cor} \label{cor:Menger_convexity_plus_alpha}
Under the assumption in Proposition \ref{prop:Menger_convexity_plus_alpha}, we have
$$
H^\ell_2(X) \cong E^2_{1, 1}
\cong \bigoplus_{\langle \varphi_0, \varphi_1 \rangle \in P_1} 
H^\ell_2(\langle \varphi_0, \varphi_1 \rangle).
$$
\end{cor}

\begin{proof}
The $E^2$-term of the spectral sequence is summarized as follows:
$$
\begin{array}{c|c|c|c|c|}
q = 2 & 0 & & & \\
\hline
q = 1 & 0 & E^2_{1, 1} & & \\
\hline
q = 0 & 0 & 0 & 0 & 0 \\
\hline
E^2_{p, q} & p = 0 & p = 1 & p = 2 & p = 3
\end{array}
$$
Then we get $E^\infty_{p, q} = E^2_{p, q}$ for any $p, q$ such that $p + q = 2$, which immediately leads to $H^\ell_2(X) = E^\infty_{1, 1} = E^2_{1, 1}$. To the direct sum decomposition of $E^2_{1, 1}$, only $H^\ell_2(\varphi)$ with $\varphi = \langle \varphi_0, \varphi_1 \rangle \in P_1 \subset \widehat{P}_1$ contributes, because $C^\ell_2(\varphi) = 0$ if $\varphi_0 = \varphi_1$.
\end{proof}


\subsection{Example}
\label{subsec:example}

Let $X$ be the metric space associated to the cycle graph $C_4$ with four vertices. Thus, $X = \{ 1, 2, 3, 4 \}$ consists of four points, and the distance is given by the shortest path length.
$$
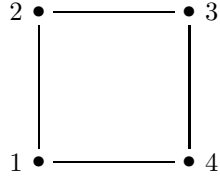
\xygraph{
!{<0cm,0cm>;<1cm,0cm>:<0cm,1cm>::}
!{(0,0)}*+{\bullet}="a" ([]!{+(-0.3,0)} {1})
!{(2,0)}*+{\bullet}="b" ([]!{+(0.3,0)} {4})
!{(0,2)}*+{\bullet}="c" ([]!{+(-0.3,0)} {2})
!{(2,2)}*+{\bullet}="d" ([]!{+(0.3,0)} {3})
"a"-"b"
"a"-"c"
"b"-"d"
"c"-"d"
}
$$
The adjacency matrix $A$ and the distance matrix $D = (d(i, j))$ are
\begin{align*}
A &=
\left(
\begin{array}{cccc}
0 & 1 & 0 & 1 \\
1 & 0 & 1 & 0 \\
0 & 1 & 0 & 1 \\
1 & 0 & 1 & 0
\end{array}
\right), &
D &=
\left(
\begin{array}{cccc}
0 & 1 & 2 & 1 \\
1 & 0 & 1 & 2 \\
2 & 1 & 0 & 1 \\
1 & 2 & 1 & 0 
\end{array}
\right).
\end{align*}
The zeta matrix $Z = (e^{-d(i, j)})$ and its inverse $M = Z^{-1}$ are
\begin{align*}
Z &=
\left(
\begin{array}{cccc}
1 & q & q^2 & q \\
q & 1 & q & q^2 \\
q^2 & q & 1 & q \\
q & q^2 & q & 1
\end{array}
\right), &
M &=
\frac{1}{(1 - q^2)^2}
\left(
\begin{array}{rrrr}
1 & -q & q^2 & -q \\
-q & 1 & -q & q^2 \\
q^2 & -q & 1 & -q \\
-q & q^2 & -q & 1
\end{array}
\right),
\end{align*}
where $q = 1/e$. The magnitude of $(X, d)$ is
$$
\mathrm{Mag} =
\frac{4}{(1 + q)^2}
= 4(1 - 2q + 3q^2 - 4q^3 + \cdots).
$$

We consider the subcomplex $C_*^3(1, 4) \subset C_*^3$ of the magnitude complex whose characteristic length is $\ell = 3$. From the $(1, 4)$-entry $M(1, 4)$ of $M$,
$$
M(1, 4)
= \sum_{\ell} q^\ell \sum_n (-1)^n\mathrm{rank}H^\ell_n(1, 4)
= \frac{-q}{(1-q^2)^2}
= -q - 2q^3 - 3 q^5 - 4q^7 - \cdots,
$$
the Euler characteristic number of $H^3_*(1, 4)$ is computed as $-2$. We have
\begin{align*}
C^3_0(1, 4) &= 0, \\
C^3_1(1, 4) &= 0, \\
C^3_2(1, 4) &= 
\Z \langle 1, 2, 4 \rangle \oplus 
\Z \langle 1, 3, 4 \rangle, \\
C^3_3(1, 4) &= \Z \langle 1, 4, 1, 4 \rangle \oplus
\Z \langle 1, 2, \overset{\circ}{1}, 4 \rangle \oplus
\Z \langle 1, \overset{\circ}{4}, 3, 4 \rangle \oplus
\Z \langle 1, \overset{\circ}{2}, \overset{\circ}{3}, 4 \rangle, \\
C^3_n(1, 4) &= 0, \quad (n \ge 4)
\end{align*}
where a circle $\circ$ above a point $x$ in a chain means that the point $x$ is smooth in the chain. The $E^1$-term is summarized as follows:
$$
\begin{array}{c|c|c|c|c|}
\hline
q = 3 & \Z & 0 & 0  & 0 \\
\hline
q = 2 & \Z^2 & \Z^2 & 0 & 0 \\
\hline
q = 1 & 0 & 0 & \Z & 0 \\
\hline
q = 0 & 0 & 0 & 0 & 0 \\
\hline
E^1_{p, q} & p = 0 & p = 1 & p = 2 & p = 3 
\end{array}
$$
The $E^2$-term is:
$$
\begin{array}{c|c|c|c|c|}
\hline
q = 3 & \Z & 0 & 0 & 0 \\
\hline
q = 2 & 0 & 0 & 0 & 0 \\
\hline
q = 1 & 0 & 0 & \Z & 0 \\
\hline
q = 0 & 0 & 0 & 0 & 0 \\
\hline
E^2_{p, q} & p = 0 & p = 1 & p = 2 & p = 3 
\end{array}
$$
Therefore $E^\infty_{p, q} = E^2_{p, q}$, although $\ell = m_X = 3$. The extension problem is trivially solved, and we get:
$$
\begin{array}{|c|c|c|c|c|c|}
\hline
& n = 0 & n = 1 & n = 2 & n = 3 & n \ge 4 \\
\hline
H^3_n(1, 4) & 0 & 0 & 0 & \Z^2 & 0 \\
\hline
\end{array}
$$
This is consistent with the computation of the Euler characteristic number of $H^3_*(1, 4)$ by using $M = Z^{-1}$. Notice that $H^3_3(1, 4) = Z^3_3(1, 4)$, and a choice of its basis is
\begin{align*}
&\langle 1, 4, 1, 4 \rangle, &
&\langle 1, 2, 1, 4 \rangle - \langle 1, 2, 3, 4 \rangle + \langle 1, 4, 3, 4 \rangle.
\end{align*}

\medskip

The direct sum decomposition in Theorem \ref{thm:direct_sum_decomposition_of_E1} is
\begin{align*}
E^1_{* + 3, 3} &= C^3_*(\langle 1, 4, 1, 4 \rangle), \\
E^1_{* + 2, 2} &= C^3_*(\langle 1, 2, 4 \rangle) 
\oplus C^3_*(\langle 1, 3, 4 \rangle), \\
E^1_{* + 1, 1} &= C^3_*(\langle 1, 4 \rangle).
\end{align*}
The complexes $C^3_*(\langle 1, 4, 1, 4 \rangle)$ and $C^3_*(\langle 1, 4 \rangle)$ have the trivial differential:
\begin{align*}
C^3_*(\langle 1, 4, 1, 4 \rangle)
&=
\left\{
\begin{array}{ll}
\Z \langle 1, 4, 1, 4 \rangle, & (* = 3) \\
0, & (* \neq 3)
\end{array}
\right.
\\
C^3_*(\langle 1, 4 \rangle)
&=
\left\{
\begin{array}{ll}
\Z \langle 1, 2, 3, 4 \rangle, & (* = 3) \\
0. & (* \neq 3)
\end{array}
\right.
\end{align*}
The chain complexes $C^3_*(\langle 1, 2, 4 \rangle)$ and $C^3_*(\langle 1, 3, 4 \rangle)$ are non-trivial, but their homology are trivial:
\begin{align*}
C^3_*(\langle 1, 2, 4 \rangle)
&=
\left\{
\begin{array}{ll}
\Z \langle 1, 2, 4 \rangle, & (* = 2) \\
\Z \langle 1, 2, 1, 4 \rangle, & (* = 3) \\
0, & (* \neq 2, 3)
\end{array}
\right.
\\
C^3_*(\langle 1, 3, 4 \rangle)
&=
\left\{
\begin{array}{ll}
\Z \langle 1, 3, 4 \rangle, & (* = 2) \\
\Z \langle 1, 2, 3, 4 \rangle, & (* = 3) \\
0. & (* \neq 2, 3)
\end{array}
\right.
\end{align*}

To see an example of the injective chain map $\iota$ in Theorem \ref{thm:tensor_product_decomposition}, we consider $C^1_*(\langle 1, 2 \rangle)$ and $C^2_*(\langle 2, 4 \rangle)$. 
\begin{align*}
C^1_*(\langle 1, 2 \rangle)
&=
\left\{
\begin{array}{ll}
\Z \langle 1, 2 \rangle, & (* = 1) \\
0. & (* \neq 1)
\end{array}
\right.
\\
C^2_*(\langle 2, 4 \rangle)
&=
\left\{
\begin{array}{ll}
\Z \langle 2, 4 \rangle, & (* = 1) \\
\Z \langle 2, 1, 4 \rangle \oplus \Z \langle 2, 3, 4 \rangle, & (* = 2) \\
0. & (* \neq 1, 2)
\end{array}
\right.
\end{align*}
The following is an example of the injective homomorphism in Theorem \ref{thm:tensor_product_decomposition} 
$$
\iota : \ C^3_*(\langle 1, 2, 4 \rangle) \to
C^1_*(\langle 1, 2 \rangle) \otimes C^2_*(\langle 2, 4 \rangle).
$$
Concretely, this is the following homomorphism
\begin{align*}
\iota(\langle 1, 2, 4 \rangle) 
&= \langle 1, 2 \rangle \otimes \langle 2, 4 \rangle, \\
\iota(\langle 1, 2, 1, 4 \rangle)
&= \langle 1, 2 \rangle \otimes \langle 2, 1, 4 \rangle.
\end{align*}
This misses $\langle 1, 2 \rangle \otimes \langle 2, 3, 4 \rangle$, and is not surjective. Further, this $\iota$ is not a quasi-isomorphism.


\section{The 3rd magnitude homology of the circle}
\label{sec:circle}

As an application of the spectral sequence, we compute the third magnitude homology of the circle, which turns out to be trivial.

\subsection{Summary of result}

Let $S^1 = \R/2\pi r \Z$ be a circle, where $r > 0$ is a positive real number. For a point $x \in S^1$, we can choose a representative $\tilde{x} \in \R$ of $x = [\tilde{x}] \in S^1$. The distance between $x_0 = [\tilde{x}_0]$ and $x_1 = [\tilde{x}_1]$ is then given by $d(x_0, x_1) = \min_{k \in \Z} \lvert x_0 - x_1 + 2\pi r k \rvert$. Accordingly, $d(x_0, x_1) = \lvert \tilde{x}_0 - \tilde{x}_1 \rvert$ if and only if $\lvert \tilde{x}_0 - \tilde{x}_1 \rvert \le \pi r$. The diameter of the circle as a metric space is
$$
\mathrm{diam}(S^1) = 
\max_{x, y \in S^1} d(x, y) = \pi r,
$$
and the distance of two points is always bounded by $\pi r$.

\begin{thm} \label{thm:thrid_homology_of_circle}
We have $H^\ell_3(S^1) = 0$ for any $\ell$.
\end{thm}

\begin{proof}
The third homology $H^\ell_3(S^1)$ receives the contributions of $E^\infty_{p, 3 - q}$ with $p = 0, 1, 2, 3$. We have $E^1_{0, 3} = 0$ by design. Applying Corollary \ref{prop:Menger_convexity_plus_alpha} to the circle $S^1$, we get $E^2_{0, 3} = 0$ for $\ell > 0$. In the case of $\ell = 0$, we trivially have $E^2_{0, 3} = 0$. By using Proposition \ref{prop:E2_21_circle} and Proposition \ref{prop:E2_12_circle} to be shown, we respectively get $E^2_{2, 1} = 0$ and $E^2_{1, 2} = 0$. Thus, we conclude $H^\ell_3(S^1) = 0$.
\end{proof}

From general properties \cite{L-S}, the magnitude homology group $H^\ell_n(S^1)$ of degree $n = 0, 1$ is readily seen for any $\ell$. In the case of degree $n = 2$, the magnitude homology is computed in \cite{K-Y} for any $\ell$, and, moreover, in the case of $\ell \le \pi r$ the magnitude homology is determined for all degree. In particular, $H^\ell_3(S^1) = 0$ for all $\ell \le \pi r$. By Theorem \ref{thm:thrid_homology_of_circle}, this vanishing turns out to hold true for all $\ell$. The following table summarizes these results, in which $\Z[S^1] = \bigoplus_{x \in S^1} \Z x$:
$$
\begin{array}{|c|c|c|c|c|c|c|}
\hline
H^\ell_n(S^1) & n = 0 & n = 1 & n = 2 & n = 3 & n = 4 & n \ge 5 \\
\hline
\ell = 0 & \Z[S^1] & 0 & 0 & 0 & 0 & 0 \\
\hline
0 < \ell < \pi r & 0 & 0 & 0 & 0 & 0 & 0 \\
\hline
\ell = \pi r & 0 & 0 & \Z[S^1] & 0 & 0 & 0 \\
\hline
\pi r < \ell < 2\pi r & 0 & 0 & 0 & 0 & ? & ? \\
\hline
\ell = 2\pi r & 0 & 0 & 0 & 0 & ? & ? \\
\hline
2\pi r < \ell & 0 & 0 & 0 & 0 & ? & ? \\
\hline
\end{array}
$$
We remark that the computation by using the spectral sequence is further developed in \cite{G}, a consequence of which determines $H^\ell_n(S^1)$ completely.

\subsection{Calculation of the spectral sequence}

Let $E^r_{p, q}$ be the smoothness spectral sequence for the magnitude homology $H^\ell_*(S^1)$. The following simple facts will be key to the calculations of $E^r_{p, q}$ here:
\begin{itemize}
\item
Let $\varphi, \varphi' \in S^1$ be distinct points. Suppose that points $x_1, \ldots, x_p$ lie on a shortest arc joining $\varphi$ to $\varphi'$ so that $d(\varphi, x_1) < d(\varphi, x_2) < \cdots < d(\varphi, x_p)$. Then $\langle \varphi, x_1, \cdots, x_p, \varphi' \rangle$ is a geodesically simple chain of length $d(\varphi, \varphi')$. 

\item
Let $x_1, \ldots, x_p, \varphi, \varphi' \in S^1$ be points such that $x_i \not < \varphi \not< \varphi'$ for $i = 1, \ldots, p$. Then there exists a point $y$ on a shortest arc joining $\varphi$ to $\varphi'$ such that $x_i \not< \varphi \not< y$ for $i = 1, \ldots, p$ and $\varphi < y < \varphi'$.

\end{itemize}

\begin{prop} \label{prop:E2_21_circle} 
$E^2_{2, 1} = 0$.
\end{prop}

\begin{proof}
Recall the direct sum decomposition
$$
E^2_{p, 1} 
= \bigoplus_{\langle \varphi_0, \varphi_1 \rangle \in \widehat{P}_1} 
C^\ell_{p+1}(\langle \varphi_0, \varphi_1 \rangle).
$$
If $\ell < d(\varphi_0, \varphi_1)$, then $E^2_{p, 1} = 0$ is clear. In the following, we shall prove that $H^\ell_3(\langle \varphi_0, \varphi_1 \rangle) = 0$ in the cases that $\ell = d(\varphi_0, \varphi_1)$ and $\ell > d(\varphi_0, \varphi_1)$.

In the case that $\ell = d(\varphi_0, \varphi_1)$, if $\ell = 0$ additionally, then  $C^\ell_{p+1}(\langle \varphi_0, \varphi_1 \rangle) = 0$ for $p \ge 0$. Hence we assume $\ell > 0$, so that $\varphi_0$ and $\varphi_1$ are distinct. Notice the obvious bound $d(\varphi_0, \varphi_1) \le \pi r$. Suppose first $d(\varphi_0, \varphi_1) < \pi r$. Then there is the unique shortest arc joining $\varphi_0$ to $\varphi_1$, and all the points $x_1, \cdots, x_p$ in any chain $\langle \varphi_0, x_1, \cdots, x_p, \varphi_1 \rangle \in C^\ell_{p+1}(\langle \varphi_0, \varphi_1)$ lie on this arc. Let us choose a point $\overline{x}$ on the arc. For a chain $\gamma = \langle \varphi_0, x_1, \cdots, x_p, \varphi_1 \rangle \in C^\ell_{p+1}(\langle \varphi_0, \varphi_1 \rangle)$, it can happen that $\overline{x} = x_i$ for an $i \in \{ 1, \ldots, p \}$. Otherwise, we have $x_i < \overline{x} < x_{i+1}$ for an $i \in \{ 0, \ldots, p \}$, where $x_0 = \varphi_0$ and $x_{p+1} = \varphi_1$. In view of the possibility of these cases, we put
$$
h(\gamma)
=
\left\{
\begin{array}{ll}
0, & (\mbox{$\overline{x} = x_i$}) \\
(-1)^i \langle \varphi_0, x_1, \cdots, x_i, \overline{x}, x_{i+1}, \cdots,
x_p, \varphi_1 \rangle. &
(x_i < \overline{x} < x_{i+1})
\end{array}
\right.
$$
The points $x_1, \ldots, x_p$ are smooth in $h(\gamma)$, because they all lie on the arc of length $\le \pi r$. Therefore $h$ defines a homomorphism 
$$
h : C^\ell_{p+1}(\langle \varphi_0, \varphi_1 \rangle)
\to C^\ell_{p+2}(\langle \varphi_0, \varphi_1 \rangle).
$$
We can verify directly $hd^1\gamma + d^1h\gamma = - \gamma$. Hence $h$ provides a chain homotopy between the trivial and the identity homomorphism on the chain complex $(C^\ell_{* + 1}(\langle \varphi_0, \varphi_1 \rangle, d^1)$. This proves $H^\ell_{p+1}(\langle \varphi_0, \varphi_1 \rangle) = 0$ for any $p$. Suppose next $d(\varphi_0, \varphi_1) = \pi r$. In this case, there is two distinct shortest arcs of length $\pi r$ that joints $\varphi_0$ to $\varphi_1$. Choosing a points $\overline{x}$ on one arc and $\overline{x}'$ on the other arc, we can define a homomorphism $h : C^\ell_p(\langle \varphi_0, \varphi_1 \rangle) \to C^\ell_{p+1}(\langle \varphi_0, \varphi_1 \rangle)$ for $p \ge 1$ in the same manner as above. This $h$ again gives a chain homotopy, and leads to $H^\ell_{p+1}(\langle \varphi_0, \varphi_1 \rangle) = 0$ for any $p > 0$.

In the case that $\ell > d(\varphi_0, \varphi_1)$, let $\langle \varphi_0, x_1, x_2, \varphi_1 \rangle \in C^\ell_3(\langle \varphi_0, \varphi_1 \rangle)$ be a chain. Because $\ell > d(\varphi_0, \varphi_1)$, this chain is a $4$-cut, namely, $\varphi_0 \not< x_1 \not< \varphi_1$ and $\varphi_0 \not< x_2 \not< \varphi_1$. Consequently, $\gamma$ is a cycle: $d^1\gamma = 0$. Since $x_1$ is smooth in the chain, it holds that $\varphi_0 < x_1 < x_2$. In addition, $d(\varphi_0, x_2) \le \pi r$. As a result, the points $\varphi_0$, $x_1$ and $x_2$ lie on a unique shortest arc joining $\varphi_0$ to $x_2$. We can choose a point $\overline{x}$ on the arc such that $\varphi_0 < \overline{x} < x_1$ and $\overline{x} \not< x_1 \not< \varphi_1$. Then $\overline{x} < x_1 < x_2$ and $\overline{x} \not< x_2 \not< \varphi_1$ follow. We thus get $\langle \varphi_0, \overline{x}, x_1, x_2, \varphi_1 \rangle \in C^\ell_4(\langle \varphi_0, \varphi_1)$ whose boundary is
$$
d^1\langle \varphi_0, \overline{x}, x_1, x_2, \varphi_1 \rangle
= - \langle \varphi_0, x_1, x_2, \varphi_1 \rangle.
$$
Therefore we also conclude $H^\ell_3(\langle \varphi_0, \varphi_1 \rangle) = 0$ in this case.
\end{proof}

\begin{prop} \label{prop:E2_12_circle}
$E^2_{1, 2} = 0$.
\end{prop}

\begin{proof}
Recall the direct sum decomposition
$$
E^1_{p, 2} 
= 
\bigoplus_{\langle \varphi_0, \varphi_1, \varphi_2 \rangle \in \widehat{P}_2} 
C^\ell_{p+2}(\langle \varphi_0, \varphi_1, \varphi_2 \rangle).
$$
We can then prove that $C^\ell_{3}(\langle \varphi_0, \varphi_1, \varphi_2 \rangle) = 0$ whenever $\varphi = \langle \varphi_0, \varphi_1, \varphi_2 \rangle \in \widehat{P}_2$ is an improper chain. It is obvious that $C^\ell_3(\varphi) = 0$ in the case that $\varphi_0 = \varphi_1 = \varphi_2$. Consider the case that $\varphi_0 = \varphi_1 \neq \varphi_2$. Then a possible $3$-chain in $C^\ell_3(\varphi)$ is of the form $\langle \varphi_0, x, \varphi_1, \varphi_2 \rangle$. In this chain, $x$ must be a smooth point. However, this is impossible under $\varphi_0 = \varphi_1$. Hence $C^\ell_3(\varphi) = 0$. The same argument applies to the case of $\varphi_0 \neq \varphi_1 = \varphi_2$. 

Henceforth we assume that $\varphi = \langle \varphi_0, \varphi_1, \varphi_2 \rangle \in P_2$ is a proper chain. Under this assumption, there appear two types of chains in $P^\ell_3(\varphi)$:
\begin{itemize}
\item[(1)]
$\langle \varphi_0, x, \varphi_1, \varphi_2 \rangle \in P^\ell_3(\varphi)$, where $\varphi_0 < x < \varphi_1$ and $x \not< \varphi_1 \not< \varphi_2$.

\item[(2)]
$\langle \varphi_0, \varphi_1, y, \varphi_2 \rangle \in P^\ell_3(\varphi)$, where $\varphi_0 \not< \varphi_1 \not< y$ and $\varphi_1 < y < \varphi_2$.
\end{itemize}
In any event, we have $\ell(\gamma) = \ell(\varphi)$ for $\gamma \in P^\ell_3(\varphi)$. This implies $C^\ell_3(\varphi) = 0$ whenever $\ell \neq \ell(\varphi)$. We thus additionally assume $\ell = \ell(\varphi)$ in the following. In view of the classification in types, we can express an element $\zeta \in C^\ell_3(\varphi)$ as
$$
\zeta = 
\sum_{i = 1}^{m} 
M_i
\langle \varphi_0, x_i, \varphi_1, \varphi_2 \rangle
+
\sum_{j = 1}^{n} 
N_i
\langle \varphi_0, \varphi_1, y_j, \varphi_2 \rangle,
$$
where the chains $\langle \varphi_0, x_i, \varphi_1, \varphi_2 \rangle$ and $\langle \varphi_0, \varphi_1, y_i, \varphi_2 \rangle$ are of type (1) and (2), respectively. Now, on a shortest arc joining $\varphi_0$ to $\varphi_1$, we can find a point $\overline{x}$ such that $\overline{x} \not< \varphi_1 \not< y_j$ for all $j$. This is possible, since $\varphi_0 \not< \varphi_1 \not< y_j$ for all $j$. We then have a chain $\langle \varphi_0, \overline{x}, \varphi_1, y_j, \varphi_2 \rangle \in C^\ell_4(\varphi)$ whose boundary is
$$
d^1\langle \varphi_0, \overline{x}, \varphi_1, y_j, \varphi_2 \rangle 
= - \langle \varphi_0, \varphi_1, y_j, \varphi_2 \rangle 
- \langle \varphi_0, \overline{x}, \varphi_1, \varphi_2 \rangle.
$$
Thus, up to a boundary, $\zeta$ agrees with the sum of chains of type (1) only
$$
\zeta =
\sum_{i = 1}^{m'} N'_i
\langle \varphi_0, x'_i, \varphi_1, \varphi_2 \rangle 
+ d^1(\mbox{$4$-chains}).
$$
On a shortest arc joining $\varphi_1$ to $\varphi_2$, we can find a point $\overline{y}$ such that $x'_i \not< \varphi_1 \not< \overline{y}$ for all $i$. Using this point, we get a chain $\langle \varphi_0, x'_i, \varphi_1, \overline{y}, \varphi_2 \rangle \in C^\ell_4(\varphi)$ whose boundary is
$$
d^1\langle \varphi_0, x'_i, \varphi_1, \overline{y}, \varphi_2 \rangle
= -\langle \varphi_0, \varphi_1, \overline{y}, \varphi_2 \rangle
-\langle \varphi_0, x'_i, \varphi_1, \varphi_2 \rangle.
$$
This allows us to express $\zeta$ as
$$
\zeta = N''
\langle \varphi_0, \varphi_1, \overline{y}, \varphi_2 \rangle
+ d^1(\mbox{$4$-chains}).
$$
Consequently, if $\zeta$ is a cycle, then it is a boundary, so that $H^\ell_3(\varphi) = 0$. Now we have $E^2_{1, 2} = \bigoplus_{\varphi} H^\ell_3(\varphi) = 0$.
\end{proof}


\end{document}